\documentclass[a4paper, reqno]{amsart}

\input{macmacro.tex}

\begin{document}

\title{Le cristal de Dieudonné des schémas en $\BF$-vectoriels}
\author{Arnaud Vanhaecke}
\begin{abstract}
Dans cet article on décrit le cristal de Dieudonné d'un schéma en groupes fini localement libre, muni d'une action vectorielle d'un corps fini $\BF$. Ces schémas en $\BF$-vectoriels apparaissent lorsqu'on considère les points de torsion d'un module $p$-divisible. Une classe particulière de schémas en $\BF$-vectoriels a été classifiée par Raynaud dans \cite{ray}, ce qui nous permet de déterminer la structure des points de torsion d'un module $p$-divisible, sous certaines conditions sur son algèbre de Lie.
\end{abstract}

\date{\today}
\maketitle

\section{Introduction}
Soit $p$ un nombre premier et $\BF$ un corps fini à $q=p^r$ éléments. On note $\Sigma = \Spec \W(\BF)$, où $W$ désigne le foncteur des vecteurs de Witt. Soit $S$ un schéma sur $\Sigma$ tel que $p$ soit localement nilpotent sur $S$. Dans cette introduction on supposera que $S$ est affine, de la forme $\Spec R$ pour $R$ une $\W(\BF)$-algèbre. Dans \cite{bbm}, les auteurs ont introduit, par la cohomologie cristalline, une théorie de Dieudonné pour les schémas en groupes finis localement libres et les groupes $p$-divisibles sur $S$. À $G$ un schéma en groupes fini localement libre (resp. $X$ un groupe $p$-divisible) sur $S$ on associe de manière fonctorielle un cristal $\BD(G)$ (resp. $\BD(X)$) sur le site cristallin $\CRIS(S/\Sigma)$. Le foncteur $\BD$ et ses propriétés de pleine fidélité ont beaucoup été étudiés (cf. \cite{deme}) mais nous ne ferons pas usage de ces résultats. Nous utiliserons seulement que dans le cas où $S$ est le spectre d'un corps parfait, le cristal de Dieudonné est équivalent au module de Dieudonné et que alors $\BD$ est une équivalence de catégories.

Un \emph{schéma en $\BF$-vectoriels} $G$ sur $S$ est un schéma en groupes fini localement libre tel que pour tout schéma $X$ sur $S$, $G(X)$ soit un espace vectoriel sur $\BF$. C'est-à-dire que $G$ est muni d'une action du groupe multiplicatif $\BF^{\times}$, satisfaisant des conditions supplémentaires. Comme $S$ est affine, $G=\Spec A$, où $A$ est une algèbre de Hopf sur $R$ munie d'une action de $\BF^{\times}$. Cette action satisfait une propriété supplémentaire : l'addition dans $\BF^{\times}\subset \BF$ est compatible avec la convolution dans $A$. Cette propriété correspond à l'axiome que pour tout $\lambda$, $\lambda' \in \BF$ et $v\in G(X)$, pour $X$ un schéma sur $S$, $(\lambda+\lambda')v =\lambda v + \lambda'v$\footnote{Noter que cet axiome est légèrement subtil car les deux sommes ne sont pas de même nature.}. Cette action détermine une graduation indexée par $\BF^{\vee}$, le groupe des caractères de $\BF^{\times}$ à valeurs dans $\bar \BQ_p$,
\begin{equation}\label{eqn:gradint}
A=\bigoplus_{\chi \in \BF^{\vee}} A_{\chi}.
\end{equation}
Il n'est pas vrai que toute graduation de type $\BF^{\vee}$ sur $A$ donne lieu à une structure de schéma en $\BF$-vectoriels sur $G$. En effet toute action de $\BF^{\times}$ sur $A$ ne s'étend pas nécessairement en une action vectorielle de $\BF$. Notons que, comme $A$ est localement libre sur $R$, il en est de même pour les $A_{\chi}$. Dans \cite{ray}, ces schémas ont été étudiés et ils sont classifiés dans le cas où pour tout $\chi \in \BF^{\vee}$ non-trivial, $A_{\chi}$ est localement libre de rang $1$ sur $R$ et $A_1 = R\oplus A_1'$, pour $A_1$ la composante isotypique du caractère trivial, tel que $A_1'$ est localement libre de rang $1$. On est donc naturellement intéressé par le rang des $A_{\chi}$. 

Pour simplifier, on suppose $S=\Spec R$ connexe. Alors le rang des $A_{\chi}$ est constant sur $R$ et on peut définir le \emph{caractère de $G$} comme
$$
\Cha_S(G) \coloneqq \sum_{\chi \in \BF^{\vee}}\rang_R(A_{\chi})[\chi]\in \BN[\BF^{\vee}].
$$
Dans cet article on calcule ce caractère à partir du cristal de Dieudonné de $G$. Le cristal $\BD(G)$ est muni d'une action de $\BF^{\times}$ et donc admet une graduation du même type que (\ref{eqn:gradint}). La relation de convolution implique que cette graduation est de la forme
$$
\BD(G) = \bigoplus_{\chi \in \BF^{+}}\BD(G)_{\chi},
$$
où $\BF^+\subset \BF^{\vee}$ est l'ensemble des caractères $\chi\in \BF^{\vee}$, tels que si on pose $\chi(0)=0$, $\chi$ est additif. Or, comme $G$ est annulé par $p$, $\BD(G)$ est localement libre modulo $p$. On peut donc parler du rang des $\BD(G)_{\chi}$.  On définit le \emph{caractère infinitésimal} de $G$ par 
$$
\cha_S(\BD(G)) \coloneqq \sum_{\chi \in \BF^+} \rang_{R/pR}(\BD(G)_{\chi})[\chi]\in \BN[\BF^+],
$$
où $\rang_{R/pR}(\BD(G)_{\chi})$ est le rang modulo $p$ de la composante $\chi$-isotypique de $\BD(G)$. Le théorème central de cet article est que le caractère et le caractère infinitésimal de $G$ sont reliés par une relation "exponentielle" : 
\begin{theorem}\label{thm:carint}
Soit $G$ un schéma en $\BF$-vectoriels sur $S$ connexe. Soit $\cha_S(\BD(G)) = \sum_{\chi\in \BF^+}n_{\chi}[\chi]$ le caractère infinitésimal de $G$ et $\Cha_S(G)$ son caractère. Alors on a 
\begin{equation}\label{eq:carforint}
\Cha_S(G) = \prod_{\chi \in \BF^+}\left ( 1+[\chi]+\dots + [\chi^{p-1}]\right )^{n_{\chi}}.
\end{equation}
\end{theorem}
On montre ce théorème en le réduisant au cas où $S$ est le spectre d'un corps parfait $k$ de caractéristique $p$. On utilise ensuite que sur le spectre d'un corps parfait le cristal de Dieudonné est équivalent au module de Dieudonné de $G$. Notons que dans ce cas, on obtient une équivalence entre schémas en $\BF$-vectoriels et modules de Dieudonné gradués sur $\BF^+$ tels que $F$ et $V$, respectivement le Frobenius et le Verschibung, définissent des endomorphismes gradués. On prouve ensuite le théorème par un calcul explicite lorsque $G$ est annulé par $V$, puis par un argument de dévissage on en déduit le théorème dans le cas où $V$ est nilpotent. On conclut finalement par un argument de dualité. Une conséquence de ce théorème est que l'on peut déduire du cristal de Dieudonné d'un schéma en $\BF$-vectoriels $G$ sur $S$ si $G$ est un schéma de Raynaud. Cette relation devrait avoir des applications dans des situations où $G$ n'est pas un schéma de Raynaud.

La formule des caractères (\ref{eq:carforint}) est apparue dans l'étude des $O_D$-modules formels spéciaux de Drinfeld (cf. \cite{drin}) en vue de la construction d'un modèle formel du premier revêtement de l'espace de Drinfeld (cf. \cite{van2}). On note $K$ une extension finie de $\BQ_p$ de degré $n=ef$ et $D$ une algèbre à division centrale sur $K$ d'invariant $1/d$, $d\geqslant 2$. On note $O_D$ l'ordre maximal de $D$ et $\Pi$ une uniformisante de $O_D$. Un groupe $p$-divisible muni d'une action de $O_D$ est appelé \emph{$O_D$-module $p$-divisible}. Maintenant, $\BF$ désigne le corps résiduel de $O_D$, c'est-à-dire que $r=fd$ avec les notations précédentes ; $S=\Spec R$ désigne toujours un schéma sur $\Sigma$ tel que $p$ est localement nilpotent sur $S$. Notons que si $X$ est un $O_D$-module $p$-divisible sur $S$ alors le schéma des $\Pi$-points de torsion $X[\Pi]$ est un schéma en $\BF$-vectoriels sur $S$. Le but de la dernière section est de calculer le caractère des points de torsion de certains modules $p$-divisibles, en particulier des $O_D$-module $p$-divisibles. 

Si $X$ est un $O_D$-module $p$-divisible sur $S$, son algèbre de Lie est munie d'une action de $\BF$. On note $\kappa$ le corps résiduel de $K$ et $\BF^+_{\kappa}\subset \BF^+$ l'ensemble des caractères $\kappa$-linéaires. L'algèbre de Lie est graduée sur $\BF^+$,
$$
\Lie(X) = \bigoplus_{\chi \in \BF^+}\Lie(X)_{\chi},
$$
et on dit que l'action de $O_D$ est \emph{stricte} sur $X$ si $\Lie(X)_{\chi}= 0$ pour tout $\chi\notin \BF_{\kappa}^+$. Si, de plus, pour tout $\chi \in \BF_{\kappa}^+$, $\Lie(X)_{\chi}$ est localement libre de rang $1$ sur $S$, alors $X$ est appelé \emph{$O_D$-module formel spécial}. Notons que dans ce cas, $X$ est de hauteur $nd^2$. L'impulsion de départ de ce travail était de trouver une preuve alternative au résultat suivant : 
\begin{proposition}
Si $X$ est un $O_D$-module formel spécial sur $S$, alors $X[\Pi]$ est un schéma de Raynaud.
\end{proposition}
Ce résultat est démontré par exemple dans \cite{wan} en utilisant que tout $O_D$-module formel spécial admet un relèvement en caractéristique zéro, puis qu'un schéma en $\BF$-vectoriels étale de rang $q$ est un schéma de Raynaud. L'intérêt est d'avoir des equations explicites pour $X[\Pi]$. À l'aide de la formule des caractères (\ref{eq:carforint}), on prouve une condition nécessaire et suffisante pour que $X[\Pi]$ soit un schéma de Raynaud.
\begin{theorem}
Soit $X$ un $O_D$-module $p$-divisible de hauteur $hnd^2$, pour $h$ un entier, sur $S$. Alors $X[\Pi]$ est un schéma de Raynaud si et seulement si $h=1$ et pour tout $\chi \in \BF^+$ on a 
$$
\rang_R(\Lie(X)_{\chi}) = \rang_R(\Lie(X)_{\chi^q}).
$$
\end{theorem}
Ces conditions sont vérifiées pour les $O_D$-modules formels spéciaux, ou plus généralement pour les $O_D$-modules formels $r$-spéciaux considérés dans \cite{raza2}.

J'aimerais remercier chaleureusement L. Fargues de m'avoir donné les clés du théorème \ref{thm:carint}  et de ses encouragements constants pour sa démonstration. Je remercie aussi S. Bartling et V. Hernandez pour leurs relectures attentives et leur intérêt pour ce travail.

\tableofcontents

\section{Schémas en $\BF$-vectoriels}
On fixe un nombre premier $p$ et un corps fini $\BF$ à $q=p^r$ éléments. Dans cette section on rappelle la définition de schéma en $\BF$-vectoriels introduite dans \cite{ray} et ses premières propriétés. On définira aussi le cristal associé à un schéma en groupes fini et on introduira les caractères du groupe et du cristal. 
\subsection{Généralités}
\subsubsection{}
Commençons par définir l'anneau de base $D$. Soit $\bar \BQ$ une cloture algébrique de $\BQ$, soit $C\subset \bar{\BQ}$  le sous-corps engendré sur $\BQ$ par $\mu_{q-1}(\bar \BQ)$, les racines $(q-1)$-ièmes de l'unité de $\bar \BQ$. Soit $D'$ l'anneau des entiers de $C$. Alors si $\xi$ est une racine $(q-1)$-ième primitive de l'unité on a $\BZ[\xi]=D'\subset C = \BQ(\xi)$. Soit $\fkp$ un idéal premier fixé de $D'$ au-dessus de $p$. Soit $D$ l'anneau obtenu à partir de $D'$ en inversant $(q-1)$ et les idéaux premiers au-dessus de $p$ différents de $\fkp$. Dans la suite de cette section $S$ sera un schéma sur $D$. Sauf mention explicite du contraire, les schémas en groupes sur $S$ considérés seront tous abéliens.

\subsubsection{}
\begin{definition}
Un foncteur $(\Sch/S)^{\circ}\rightarrow (\Vect/\BF)$ de la catégorie des schémas sur $S$ vers la catégorie des $\BF$-espaces vectoriels sera appelé \emph{schéma en $\BF$-vectoriels} s'il est représentable par un schéma fini, localement libre et de présentation finie. C'est donc un schéma en groupes abéliens muni d'une action de $\BF$.
\end{definition}
Soit $G$ un schéma en $\BF$-vectoriels. Comme $G$ est fini sur $S$, il est relativement affine sur $S$ i.e. il existe un $\sO_S$-module $\sA$ tel que $G=\Spec_S\sA$. De plus, $\sA$ est une $\sO_S$-algèbre de Hopf qui est munie d'une co-action de $\BF$. Notons :
\begin{itemize}
\item[$\bullet$] $\mu \colon \sA \otimes_S \sA \rightarrow \sA $ la multiplication,
\item[$\bullet$] $ \Delta \colon \sA \rightarrow \sA \otimes_S \sA$ la comultiplication,
\item[$\bullet$] $\sI$ l'idéal d'augmentation de $\sA$, qui est le noyau de la co-unité,
\item[$\bullet$] $[\lambda] \colon \sA \rightarrow \sA$ pour tout $\lambda \in \BF$ la co-action de $\BF$. Pour tout $\lambda, \lambda' \in \BF$ on a les relations 
\begin{equation}\label{rel:conv}
[\lambda][\lambda']=[\lambda\lambda']\text{ et }[\lambda +\lambda']=\mu\circ([\lambda]\otimes[\lambda'])\circ \Delta.
\end{equation}
\end{itemize}
Pour $n\geqslant 2$ un entier, la co-associativité (resp. l'associativité) permet de définir la $n$-ième itération de $\Delta$ (resp. $\mu$) comme étant une application $\Delta_n : \sA \otimes_S \dots \otimes_S \sA \rightarrow \sA$ (resp. $\mu_n : \sA \rightarrow \sA \otimes_S \dots \otimes_S \sA$) où $\sA$ apparaît $n$ fois.

La proposition suivante est immédiate : 
\begin{proposition}
Tout schéma en $\BF$-vectoriels est annulé par $p$ et est localement de rang une puissance de $p$.
\end{proposition}
\qed

Rappelons que par \cite[Théorème 2.4.2 exp. $VII_B$]{sga3}, si $S$ est le spectre d'un corps, la catégorie des schémas en groupes finis sur $S$ est une catégorie abélienne, il en est donc de même pour la catégorie des schémas en $\BF$-vectoriels sur $S$. De plus, pour un schéma $S$ sur $D$ quelconque la catégorie des schémas en groupes finis localement libres n'est pas abélienne mais elle s'identifie à une sous-catégorie exacte (au sens de Quillen) de la catégorie des faisceaux fppf sur $S$ (cf. \cite[Exposé V]{sga3}. Ainsi, dans ce contexte, on peut parler de suite exacte et de foncteur exact. Ces considérations s'appliquent donc aussi à la catégorie des schémas en $\BF$-vectoriels sur $S$.
\subsubsection{}
Supposons dans ce paragraphe que $S = \Spec R$ où $R$ est une $D$-algèbre. On sait que la catégorie des schémas en groupes finis localement libres et de présentation finie sur $S$ est équivalente à la catégorie des $R$-algèbres de Hopf finies localement libres et de présentation finie, commutatives et co-commutatives. Les relations \ref{rel:conv} nous indiquent alors que se donner un schéma en $\BF$-vectoriels $G$ sur $S$ est équivalent à se donner une $R$-algèbre de Hopf $A$ et un morphisme d'anneaux commutatifs $\iota \colon \BF \rightarrow \End(A)$. Rappelons que $\End(A)$ est l'anneau des endomorphismes de l'algèbre de Hopf $A$ où l'addition est donnée par la convolution
$$
\forall f,g \in \End(A),\ f+g = \mu\circ(f\otimes g)\circ \Delta,
$$
 et le produit est le produit usuel. On obtient ainsi une équivalence de catégories :
\begin{proposition}\label{prop:eqaff}
Supposons que $S=\Spec R$ pour $R$ une $D$-algèbre. Le foncteur qui à $G$, un schéma en $\BF$-vectoriels, associe la paire $(A,\iota)$ induit une équivalence de catégories
$$
\{G \text{ schéma en $\BF$-vectoriels sur }S\}\xleftrightarrow {} \left \{\begin{array}c
											(A,\iota),\ A \text{ une R-Algèbre de Hopf}\\
											\text{finie localement libre, commutative et}\\
											\text{co-commutative, de présentation finie, }\\
											\iota \colon \BF \rightarrow \End(A), \text{ morphisme d'anneaux.}  \\
											\end{array}\right \}.
$$
\end{proposition}
\begin{remark}
Pour $S$ un schéma quelconque sur $D$, on obtient sans peine une version faisceautique du théorème précédent : se donner un schéma en $\BF$-vectoriels $G$ sur $S$ est équivalent à se donner une $\sO_S$-algèbre $\sA$, avec les mêmes conditions de finitudes, munie d'un morphisme de faisceaux d'anneaux $\BF \rightarrow \End(\sA)=\End(G)$.
\end{remark}

\subsubsection{}
Soit $\BF^{\vee} = \Hom(\BF^{\times},\mu_{q-1}(D))$ les caractères multiplicatifs de $\BF^{\times}$. Notons que comme l'ordre de $\BF^{\times}$ est $(q-1)$, qui est inversible dans $D$ par construction, $\BF^{\vee}$ est l'ensemble des représentations linéaires de $\BF^{\times}$ à coefficients dans $D$. En particulier, 
$$
D[\BF^{\times}] = \prod_{\chi\in \BF^{\vee}}D_{\chi},
$$
où $D_{\chi}$ est la représentation de $\BF^{\times}$ sur $D$ de rang $1$ associée à $\chi \in \BF^{\vee}$. Pour tout $\chi \in \BF^{\vee}$ on note $p_{\chi} \colon D[\BF^{\times}] \rightarrow D_{\chi}$ le projecteur associé à cette décomposition.

Soit $G$ un schéma en $\BF$-vectoriels et $\sA$ la $\sO_S$-algèbre de Hopf associée. Alors $\sA$ est en particulier un faisceau de représentations de $\BF^{\times}$, i.e., par le morphisme structural $D \rightarrow \sO_S$, un faisceau en $D[\BF^{\times}]$-modules. Ainsi, par ce qui précède
$$
\sA = \bigoplus_{\chi \in \BF^{\vee}} \sA_{\chi},\ \sA_{\chi} = p_{\chi}(\sA).
$$
Il est de plus clair que cette décomposition est une graduation de $\sO_S$-algèbre de Hopf, i.e. elle est compatible au produit et au co-produit.
On a montré la proposition suivante : 
\goodbreak
\begin{proposition}
Soit $G$ un schéma en $\BF$-vectoriels sur $S$, alors sa $\sO_S$-algèbre de Hopf $\sA$ est munie d'une $\BF^{\vee}$-graduation de $\sO_S$-algèbre de Hopf qui détermine la co-action de $\BF$.
\end{proposition}
\qed

\begin{remark}
Notons que comme $\BF^{\times}$ agit trivialement sur $\sO_S$, la graduation sur le noyau de la co-unité $\sA\rightarrow \sO_S$ est induite par la graduation sur $\sA$. Ainsi la graduation de l'idéal d'augmentation est de la forme
$$
\sI = \bigoplus_{\chi \in \BF^{\vee}}\sI_{\chi}, \ \sI = p_{\chi}(\sI).
$$
C'est la graduation considérée dans \cite{ray}.
\end{remark}
\goodbreak
\subsubsection{}
Soit $G$ un schéma en $\BF$-vectoriels sur $S$ et $\sA$ sa $\sO_S$-algèbre de Hopf graduée associée. On suppose pour simplifier les notations que $S$ soit connexe ; on va s'intéresser au rang des $\sA_{\chi}$ que l'on notera $\rang_S(\sA_{\chi})$.
\begin{definition}
On conserve les notations et hypothèses précédentes. On définit \emph{le caractère de $G$} par
$$
\Ch_S(G) = \sum_{\chi \in \BF^{\vee}}\rang_S(\sA_{\chi})[\chi]\in \BN[\BF^{\vee}].
$$
\end{definition}
Cette définition a toujours un sens si $G$ est un schéma en groupe fini et plat muni d'une action de $\BF^{\times}$. On remarque que le caractère est invariant par changement de base : pour $S'\rightarrow S$ un morphisme on a 
$$
\Ch_S(G) = \Ch_{S'}(G\otimes_S S').
$$
Notons de même que pour tout $\chi \in \BF^{\vee}$ la fonction de rang
$$
s\in S \mapsto \dim_{\kappa(s)}(\sA_{\chi}\otimes_S\kappa(s)),
$$
où $\kappa(s)$ est le corps résiduel en $s\in S$, est constante sur $S$ égale au rang de $\sA_{\chi}$ comme $\sO_S$-module. Ainsi pour tout point géométrique $\Spec (k) \hookrightarrow S$ on a 
$$
\Ch_S(G) = \Ch_k(G\otimes_S k).
$$
Ceci nous permettra de restreindre le calcul du caractère au cas où $S$ est le spectre d'un corps parfait.
\subsubsection{}\label{sssub:dual}
Soit $G$ un schéma en $\BF$-vectoriels sur $S$. Soit $G^D=\sHom(G,\BG_m)$ le dual de Cartier du schéma en groupe fini et plat sous-jacent. Alors comme $\BF^{\times}$ est commutatif, $G^D$ est muni d'une action de $\BF^{\times}$ définie par :
$$
\forall f\in G^D, \lambda \in \BF, x\in G,\ (\lambda f)(x) = f(\lambda x)
$$
De plus, il est clair que cette action définit naturellement une structure de $\BF$-vectoriels sur $G$. On obtient 
$$
\Ch_S(G) = \Ch_S(G^D).
$$
Il est important de noter que ceci est faux si on munit $G^D$ de l'action contragrédiente de $\BF^{\times}$.
\subsection{Cristaux de Dieudonné gradués}
\subsubsection{}
\begin{definition}
Soit $\chi \in \BF^{\vee}$. On dit que $\chi$ est \emph{primitif} s'il est additif. On note $\BF^{+}\subset \BF^{\vee}$ l'ensemble des caractères primitifs de $\BF^{\vee}$.
\end{definition}

\begin{remark}\label{rem:primfond}
\ 
\begin{itemize}
\item[$\bullet$] On a un isomorphisme d'anneau $D' \cong \BZ[\BF^{\vee}]$ qui identifie les racines primitives de l'unité aux caractères primitifs. C'est pourquoi on préfère cette terminologie à celle de "fondamental" dans \cite{ray}.
\item[$\bullet$]Si $\chi$ un caractère primitif, alors $\chi^p$ est primitif et $\BF^+ = \{\chi,\chi^p,\chi^{p^2},\dots, \chi^{p^{r-1}}\}$. On a de plus un isomorphisme\footnote{Les isomorphismes sont donnés par les identifications $\BF^{\times} \cong \mu_{q-1}(D)\subset C$.} non-canonique $\BF^{+}\cong \Gal(\BF/\BF_p)$, plus précisément, $\BF^{+}$ est un espace principal homogène sous $\Gal(\BF/\BF_p)$.
\end{itemize}
\end{remark}
On fixe un caractère primitif $\chi_1 \in \BF^{+}$ et on note $\chi_i = \chi_1^{p^i}$ pour tout entier $i\in\llbracket 0, r-1 \rrbracket$. Tout caractère non-trivial $\chi \in \BF^{\vee}$ a un unique développement $p$-adique de la forme
$$
\chi = \prod_{i=0}^{p-1}\chi_i^{a_i}, \text{tel que $a_i \in \llbracket 0, p-1 \rrbracket$, pour tout $i\in\llbracket 0, r-1 \rrbracket$}.
$$
Le caractère trivial $\chi_0$ s'écrit de deux manières différentes : on a $\chi_0 = \prod_{i=0}^{p-1}\chi_i^{p-1}=\chi_1^0$.

\subsubsection{}
On introduit quelques notations dans ce paragraphe, suivant principalement \cite{bbm}. Soit $\Sigma = \Spec \W(\BF)$, c'est un schéma sur $D$. Supposons que $S$ est un schéma sur $\Sigma$ tel que $p$ est (Zariski)-localement nilpotent sur $S$. On note $\sigma \colon \W(\BF) \rightarrow \W(\BF)$ l'automorphisme de Frobenius. On considère la catégorie cristalline $\CRIS(S/\Sigma)$ munie de la topologie fppf. Soit $G$ un schéma en groupes fini et plat sur $S$, alors on définit un préfaisceau sur $\CRIS(S/\Sigma)$ en posant
$$
\underline G (U,T,\delta) = \Hom(U,G),
$$
pour tout objet $(U,T,\delta)$ de $\CRIS(S/\Sigma)$. Le préfaisceau $\underline G$ est en effet un faisceau fppf (cf. \cite[Corolaire 1.1.8]{bbm}). Soit $\sO_{S/\Sigma}$ le faisceau structural usuel sur le site $\CRIS(S/\Sigma)$. On considère le \emph{cristal de Dieudonné (covariant)} de $G$,
$$
\BD(G) = \sExt^1_{S/\Sigma}(\underline {G^D} , \sO_{S/\Sigma})^{(\sigma^{-1})},
$$
qui est un cristal en $\sO_{S/\Sigma}$-modules localement de présentation finie (cf. \cite[Corollaire 3.1.3]{bbm}). Notez qu'on change $G$ par $G^D$ par rapport à la définition de \cite{bbm} pour avoir un foncteur covariant, et que l'on tord l'action de $\W(\BF)$ par $\sigma^{-1}$. Cette dernière modification nous évitera d'introduire des décalages dans la suite. La dualité de Cartier $G\rightsquigarrow G^D = \sHom_S(G,\BG_m)$ est une anti-équivalence exacte des schémas en groupes finis et localement libres sur $S$, donc cette modification est innocente pour l'application des résultats de \cite{bbm} dans la suite.

Si $G$ est annulé par $p$, de rang constant $p^d$, 
par exemple si $G$ un schéma en $\BF$-vectoriels sur une base $S$ connexe,
alors $\BD(G)$ est un $\sO_{S/\Sigma}/p\sO_{S/\Sigma}$-module localement libre de rang fini $d$ (cf. \cite[Proposition 4.3.1]{bbm}) et de plus le foncteur $G\rightsquigarrow \BD(G)$ restreint à la catégorie des schémas en $\BF$-vectoriels est exact.

\subsubsection{}
Supposons que $G$ est un schéma en $\BF$-vectoriels sur $S$. Alors l'action de $\BF^{\times}$ induit une structure de $\sO_{S/\Sigma}$-module $\BF^{\times}$-équivariant sur $\BD(G)$, ou de manière équivalente, une graduation de type $\BF^{\vee}$ sur $\BD(G)$. Plus précisément, par la fonctorialité de $\BD$ l'endomorphisme
$$
[\lambda] \colon G\rightarrow G,
$$
pour $\lambda \in \BF^{\vee}$, induit un morphisme $\lambda \cdot \colon \BD(G) \rightarrow \BD(G)$. De même pour tout $\chi \in \BF^{\vee}$, on peut définir un endomorphisme idempotent $p_{\chi}\cdot \colon \BD(G) \rightarrow \BD(G)$. On a la proposition suivante :
\begin{proposition}\label{prop:gradcris}
Soit $G$ un schéma en $\BF$-vectoriels sur $S$. Alors
$$
\BD(G) = \bigoplus_{\chi\in \BF^+}\BD(G)_{\chi},\ \text{ où } \ \BD(G)_{\chi} = p_{\chi}\cdot\BD(G).
$$
Ainsi, $\BD$ est un foncteur de la catégorie des schémas en $\BF$-vectoriels sur $S$ dans la catégorie des $\sO_{S/\Sigma}/p\sO_{S/\Sigma}$-modules gradués sur $\BF^+$.
\end{proposition}
\goodbreak
\begin{proof}
Par ce qui précède, on a 
$$
\BD(G) = \bigoplus_{\chi\in \BF^{\vee}}\BD(G)_{\chi},\ \text{ où } \ \BD(G)_{\chi} = p_{\chi}\cdot\BD(G).
$$
Il suffit de montrer que pour $\chi\in \BF^{\vee}\backslash \BF^+$, $p_{\chi}\cdot\BD(G) = 0$. Or, le foncteur $\BD$ est additif et donc on obtient un morphisme d'anneaux
$$
\End(G) \rightarrow \End(\BD(G)).
$$
Soit $\lambda, \lambda' \in \BF$. Comme le morphisme ci-dessus associe $\lambda\cdot$ à $[\lambda]$ et comme il respecte l'addition, on a 
$$
(\lambda +\lambda')\cdot=\lambda\cdot + \lambda' \cdot.
$$
Si $\chi \in \BF^{\vee}$ alors $\lambda\cdot \colon \BD(G)_{\chi}\rightarrow \BD(G)_{\chi}$ équivaut à la multiplication par $\chi(\lambda)$. Ainsi, si $\BD(G)_{\chi}$ est non-nul, $\chi(\lambda+\lambda') = \chi(\lambda)+\chi(\lambda')$, et donc $\chi\in \BF^+$.
\end{proof}
On pose $S_{\BF} = S\otimes_{\BZ_p}\W(\BF)$. C'est naturellement un schéma sur 
$$
\W(\BF) \otimes_{\BZ_p} \W(\BF) = \prod_{\chi\in \BF^+}\W(\BF)_{\chi},
$$
donc, comme $S$ est un schéma sur $\W(\BF)$, on obtient
$$
S_{\BF} = \bigsqcup_{\chi\in\BF^+}S_{\chi} \text{ où } S_{\chi}=S\otimes_{\W(\BF)}\W(\BF)_{\chi}.
$$
Ainsi,
$$
 \sO_{S_{\BF}/\Sigma} = \bigoplus_{\chi\in\BF^+}(\sO_{S/\Sigma})_{\chi}.
$$
La proposition précédente peut alors s'énoncer de la manière suivante : 
\begin{corollary}\label{cor:gradeq}
Le foncteur $\BD$ induit un foncteur exact de la catégorie des schémas en $\BF$-vectoriels dans la catégorie des cristaux en $\sO_{S_{\BF}/\Sigma}/p\sO_{S_{\BF}/\Sigma}$-modules localements libres de rang fini et de présentation finie.
\end{corollary}
\qed

\begin{remark}
Ce foncteur pour les schémas en $\BF$-vectoriels possède les mêmes propriétés que le foncteur sous-jacent pour les schémas en groupes finis et localement libres, i.e pour certaines restrictions sur $S$ il est fidèle, pleinement fidèle ou une équivalence de catégories\footnote{Par exemple si $S$ est de caractéristique $p$, localement noetherien d'anneaux locaux à intersection complète, \BD est fidèle \cite{deme} et si $S$ est le spectre d'un corps parfait c'est une équivalence comme on le rappellera.}.
\end{remark}

Supposons $S$ connexe et fixons $G$ un schéma en $\BF$-vectoriels sur $S$. Alors comme $\BD(G)$ est localement libre sur $\sO_{S/\Sigma}/p\sO_{S/\Sigma}$, on peut définir $\rang_{\sO_{S/\Sigma}/p\sO_{S/\Sigma}}(\BD(G)_{\chi})$, le rang de la composante $\chi$-isotypique, pour tout $\chi \in \BF^+$.

\begin{definition}\label{def:infcar}
Soit $\sF$ un $\sO_{S_{\BF}/\Sigma}/p\sO_{S_{\BF}\Sigma}$-module localement libre de rang fini. On définit le \emph{caractère} de $\sF$ par
$$
\cha_S(\sF) = \sum_{\chi\in \BF^+}\rang_{\sO_{S/\Sigma}/p\sO_{\Sigma}}(\sF_{\chi})[\chi]\in \BN[\BF^+].
$$

\end{definition}
On considèrera en particulier $\cha_S(\BD(G))$. 
\subsubsection{}\label{sssec:bla}
Dans ce paragraphe on suppose que $S=\Spec(k)$ pour $k$ un corps parfait de caractéristique $p$ contenant $\BF$. Soit $G$ un schéma en groupes fini et localement libre sur $S$. On pose
$$
D(G) = \Gamma(S/\Sigma,\BD(G)).
$$
C'est un $\W(k)$-module de longueur finie (cf. \cite[Proposition 4.2.10]{bbm}) muni de deux morphismes $F \colon D(G) \rightarrow D(G)^{(\sigma)}$ et $V \colon D(G)^{(\sigma)}\rightarrow D(G)$ tels que $FV=VF=p$. C'est le dual du module de Dieudonné de $G$ au sens de \cite[Chapitre IV]{font} d'après \cite[Théorème 4.2.14]{bbm}, i.e. c'est le module de Dieudonné covariant de $G$. On sait que $D$ établit une équivalence entre la catégorie des schémas en groupes finis et les modules de Dieudonné sur $\W(k)$ (cf. \cite[IV 1.3]{font}). 

Si $G$ est un schéma en $\BF$-vectoriel alors d'après le corollaire \ref{cor:gradeq}, $D(G)$ est naturellement un module de Dieudonné gradué sur $\BF^+$ au sens suivant : 
\begin{definition}
Soit $M$ un $W(k)$ module muni de deux morphismes $F \colon M \rightarrow M^{(\sigma)}$ et $V \colon M^{(\sigma)}\rightarrow M$ tels que $FV=VF=0$. On dit que $M$ est un \emph{module de Dieudonné gradué} sur $\BF^+$ si on a de plus une graduation
$$
M = \bigoplus_{\chi \in \BF^+} M_{\chi}
$$
telle que $F$ et $V$ se restreignent en
$$
F \colon M_{\chi} \rightarrow M_{\chi^p}, \ V \colon M_{\chi^p} \rightarrow M_{\chi},
$$
pour tout $\chi \in \BF^+$.
\end{definition}
Ainsi l'équivalence de catégories pour les schémas en groupes finis et localement libres nous donne la caractérisation suivante :
\begin{proposition}
Le foncteur $D$ induit une équivalence de catégories 
$$
\{\text{Schémas en $\BF$-vectoriels sur }k\}\xleftrightarrow {}  \{ \W(k)\text{-modules de Dieudonné de longueur finie gradués sur } \BF^+\}.
$$
\end{proposition}
\qed

\subsubsection{}\label{sssec:redp}
Soit $M$ un $\W(k)$-module de Dieudonné gradué sur $\BF^+$. De même que dans la définition \ref{def:infcar}, on peut définir le caractère de $M$ : 
$$
\cha_k(M) = \sum_{\chi\in\BF^+}\dim_k(M_{\chi}) [\chi]\in\BN[\BF^+].
$$
Soit à nouveau $S$ un schéma connexe sur $\Sigma$ tel que $p$ est localement nilpotent sur $S$. Soit $s\colon \Spec(k)\hookrightarrow S$ un point géométrique. Soit $G$ un schéma en $\BF$-vectoriels sur $S$, alors comme $\BD(G)$ est un cristal et qu'il est compatible aux changements de base, on a pour tout élément $(U,T,\delta)$ de $\CRIS(S/\Sigma)$ un isomorphisme
$$
D(G_k)\cong \BD(G_k)_{\Spec k} \cong \BD(G)_{(U,T,\delta)}\otimes_{\sO_T} k. 
$$
Ainsi, comme $\BD(G)$ est un $\sO_{S/\Sigma}/p\sO_{S/\Sigma}$ module libre de rang fini, on a
$$
\cha_S(\BD(G)) = \cha_k(D(G_k)).
$$
Ceci nous permettra de ramener le calcul du caractère sur une base générale au calcul du caractère sur un corps parfait. Le cristal de Dieudonné ne se comporte pas très bien par rapport à la dualité de Cartier. Néanmoins, la description précédente nous donne le résultat suivant :
\begin{lemma}\label{lem:dual}
Soit $G$ un schéma en $\BF$-vectoriels sur $S$ connexe. Alors
$$
\cha_S(\BD(G)) = \cha_S(\BD(G^D)).
$$
\end{lemma}
\begin{proof}
D'après ce qui précède, on peut supposer que $S= \Spec(k)$ où $k$ est un corps parfait de caractéristique $p$. On pose $\W_{\infty}= \W(k)[\frac 1 p]/\W(k)$ le module dualisant et on définit pour tout $\W(k)$-module de Dieudonné gradué sur $\BF^+$, le $\W(k)$-module
$$
M^*=\Hom_{\W(k)}(M,\W_{\infty}).
$$
Comme dans le cas des schémas en $\BF$-vectoriels, $M^*$ est naturellement muni d'une structure de $\W(k)$-module gradué\footnote{Encore une fois l'action de $\BF^{\times}$ sur $M$ n'est pas l'action contragrédiente, qui ne définit pas une graduation sur $\BF^+$, mais l'action de l'opposé de $\BF^{\times}$ qui lui est égale par commutativité.} et $\cha_k(M) = \cha_k(M^{*})$. Or, d'après \cite[Chapitre IV, \S5, Corollaire 2]{font}, on a un isomorphisme fonctoriel en $G$, $D(G^D) \cong D(G)^*$. D'où le résultat.
\end{proof}


\section{Formule des caractères et $\BF$-exponentielle}
Dans cette section on démontre la formule des caractères. On commence par définir la $\BF$-exponentielle.
\begin{definition}
On appelle  \emph{$\BF$-exponentielle} l'application multiplicative $\exp_{\BF} \colon \BN[\BF^{+}]\rightarrow \BN[\BF^{\vee}]$ définie pour $f = \sum_{\chi \in \BF^{+}} n_{\chi}[\chi]\in \BN[\BF^{+}]$ par 
$$
\exp_{\BF}(f)=\prod_{\chi \in \BF^{+}}(1+[\chi]+[\chi^2]+\dots +[\chi^{p-1}])^{n_{\chi}}.
$$
\end{definition}
Par \emph{multiplicatif} on entend que $\exp_{\BF}(f+g)=\exp_{\BF}(f)\exp_{\BF}(g)$ pour tout $f,g\in \BN[\BF^{+}]$. Le but est de démontrer le théorème suivant :
\begin{theorem}\label{thm:expcarg}
Soit $G$ un schéma en $\BF$-vectoriels sur un $\W(\BF)$-schéma connexe $S$ tel que $p$ est localement nilpotent sur $S$. Alors
$$
\Cha_S(G) = \exp_{\BF}(\cha_S(\BD(G))).
$$
\end{theorem}
D'après ce qu'on a expliqué au paragraphe \ref{sssec:redp}, ce théorème est équivalent au théorème suivant :

\begin{theorem}\label{thm:expcar}
Soit $G$ un schéma en $\BF$-vectoriels sur un corps parfait $k$ de caractéristique $p$ tel que $\BF\subset k$. Alors
$$
\Cha_k(G) = \exp_{\BF}(\cha_k(D(G))).
$$
\end{theorem}
Dans toute cette section on suppose que $G$ est un schéma en $\BF$-vectoriels sur un corps parfait $k$ de caractéristique $p$ tel que $\BF\subset k$. On notera toujours $\sigma \colon k \rightarrow k$ l'automorphisme de Frobenius.
\goodbreak
\subsection{Multiplicativité du caractère}
\subsubsection{}
Le caractère des modules de Dieudonné gradués sur $\BF^+$ est additif, plus précisément, on a directement :
\begin{lemma}
Soit
$$
0\rightarrow M' \rightarrow M \rightarrow M'' \rightarrow 0
$$
une suite exacte de $\W(k)$-modules de Dieudonné gradués sur $\BF^+$, en particulier les morphismes sont gradués. Alors
$$
\cha_k(M) = \cha_k(M')+\cha_k(M'').
$$
\end{lemma}
\qed

Ainsi, d'après le théorème \ref{thm:expcar}, on devrait obtenir que $\Cha_k$ est multiplicatif. La démonstration du théorème \ref{thm:expcar} repose sur ce fait ; on montre donc dans cette partie la proposition suivante :
\begin{proposition}\label{prop:multcar}
Soit
$$
0\rightarrow G'\rightarrow G \rightarrow G''\rightarrow 0
$$
une suite exact de schémas en $\BF$-vectoriels\footnote{Notez que les morphismes de schémas en $\BF$-vectoriels sont des morphismes de schémas en groupes qui respectent l'action de $\BF$.}. Alors
$$
\Cha_k(G) = \Cha_k(G')\Cha_k(G'').
$$
\end{proposition}

\subsubsection{}
Soit $G$ un schéma en $\BF$-vectoriels sur $k$ et $A$ son algèbre de Hopf. On suppose que $G$ est connexe, c'est-à-dire que $A$ est une $k$-algèbre artinienne locale. Rappelons qu'on a une graduation
$$
A=\bigoplus_{\chi\in \BF^{\vee}}A_{\chi}.
$$
\begin{definition}
Soit $M$ un $A$-module. Supposons que $M$ est gradué sur $\BF^{\vee}$. Alors on dira que $M$ est un $A$-\emph{module gradué} si pour tout $\chi,\chi'\in \BF^{\vee}$ on a $A_{\chi}\cdot M_{\chi'}\subset M_{\chi\chi'}$. On dira de plus qu'un élément de $M$ est \emph{homogène} de type $\chi \in \BF^{\vee}$ s'il appartient à $M_{\chi}$.
\end{definition}
\begin{example}
Pour $\chi \in \BF^{\vee}$, on note $A(\chi)$ le $A$-module libre de rang $1$ muni d'une graduation sur $\BF^{\vee}$ telle que pour tout $\mu \in \BF^{\vee}$ on ait $A(\chi)_{\mu} = A_{\chi^{-1} \mu}$. C'est un $A$-module gradué. On peut définir de même $M(\chi)$ pour tout $A$-module gradué $M$. En particulier, on a les $k(\chi)$ où $k$ est muni de la graduation triviale.
\end{example}
\begin{proposition}\label{prop:relat}
Soit $M$ un $A$-module gradué et libre de rang $d$ sur $A$. Alors on a une décomposition
$$
M=\bigoplus_{\chi \in \BF^{\vee}} A(\chi)^{d_{\chi}},
$$
telle que $\sum_{\chi\in\BF^{\vee}}d_{\chi}= d$.
\end{proposition}

\begin{proof}
Il suffit de montrer qu'il existe une $A$-base de $M$, constituée d'éléments homogènes. On considère le produit tensoriel $M_k = M\otimes_{A}k$ pour la co-unité $A\rightarrow k$ ; c'est un $k$-module gradué sur $\BF^{\vee}$,
$$
M_k=\bigoplus_{\chi \in \BF^{\vee}} k(\chi)^{d_{\chi}}.
$$
De plus, $M\rightarrow M_k$ est surjectif et respecte la graduation car $k\subset A_{\chi_0}$, c'est-à-dire que $M_{\chi}\rightarrow k(\chi)^{d_{\chi}}\subset M_k$ est surjectif pour tout $\chi \in \BF^{\vee}$. Ainsi, pour tout $\chi \in \BF^{\vee}$ on peut fixer un ensemble à $d_{\chi}$ éléments homogènes $\fkB_{\chi}\subset M_{\chi}$ tels que leurs images dans $M_k$ forment une $k$-base de $k(\chi)^{d_{\chi}}$. On pose $\fkB= \cup \fkB_{\chi}$, c'est un ensemble à $d$ éléments que nous notons $v_1,\dots, v_d$. Par le lemme de Nakayama, $\fkB$ est une famille génératrice de $M$ en tant que $A$-module. On a donc une surjection de $A$-modules
\begin{eqnarray*}\label{eq:decomp}
 A^d & \rightarrow & M \\
(x_1,\dots, x_d)&\mapsto &\sum_{i=1}^dx_i\cdot v_i.
\end{eqnarray*}
Or, c'est une surjection entre $k$-espaces vectoriels de même dimension et donc c'est un isomorphisme de $A$-modules. Ce qui montre que $\fkB$ est une $A$-base homogène de $M$.
\end{proof}

\subsubsection{}
On déduit de la proposition précédente une formule pour les dimensions des composantes isotypiques.
\begin{corollary}\label{cor:relat}
On conserve les notations et hypothèses de la proposition \ref{prop:relat}. On a alors 
$$
\dim_k(M_{\chi})=\sum_{\mu \mu'=\chi}d_{\mu}\dim_k(A_{\mu'}),
$$
où la somme porte sur les $\mu,\mu' \in \BF^{\vee}$ tels que $\mu\mu'=\chi$.
\end{corollary}
\qed

Soit $G_1$, $G_2$ deux schémas en $\BF$-vectoriels sur $k$ tels que $G_1$ soit connexe. Soit $A_1$, $A_2$ leurs algèbres de Hopf respectives. Supposons que l'on a une une surjection $G_2\rightarrow G_1\rightarrow 0$, alors d'après \cite[Exposé VII$_B$ 2.4]{sga3} ce morphisme est plat. Ainsi $G_2$ est localement libre sur $G_1$ car $G_1$ est noetherien. Or, $A_1$ est Artinien donc $A_2$ est libre sur $A_1$ et on peut appliquer la proposition précédente, qui donne une décomposition 
\begin{equation}\label{eq:decisoc}
A_2=\bigoplus_{\chi \in \BF^{\vee}} A_1(\chi)^{d_{\chi}}.
\end{equation}
\begin{proposition}\label{prop:multcarco}
Soit $0\rightarrow G_1\rightarrow G_2\rightarrow G_3\rightarrow 0$ une suite exacte de schémas en $\BF$-vectoriels.  Supposons que $G_3$ soit connexe. Alors
$$
\Cha_k(G_2)=\Cha_k(G_1)\Cha_{k}(G_3).
$$
\end{proposition}
\begin{proof}
Soit $A_i$ la $k$-algèbre de Hopf associée à $G_i$, pour  $i\in\{1,2,3\}$. D'après ce qui précède, on a la décomposition (\ref{eq:decisoc}). Or $G_1 = G_2 \times_{G_3}\Spec k$, pour le morphisme unité $\Spec k \rightarrow G_3$, donc
$$
A_1 = \bigoplus_{\chi \in \BF^{\vee}} k(\chi)^{d_{\chi}},
$$
en particulier $d_{\mu} = \dim_k(A_{3,\mu})$ pour tout $\mu \in \BF^{\vee}$. Par le corollaire \ref{cor:relat} on a pour tout $\chi \in \BF^{\vee}$ la relation
$$
\dim_k(A_{2,\chi})=\sum_{\mu \mu'=\chi}\dim_k(A_{3,\mu})\dim_k(A_{1,\mu'}),
$$
ce qui prouve la proposition.
\end{proof}
\subsubsection{}
On est maintenant en mesure de conclure l'argument pour montrer la proposition \ref{prop:multcar}, i.e. d'enlever dans la proposition précédente l'hypothèse que $G_3$ soit connexe. Soit 
$$
0\rightarrow G'\rightarrow G \rightarrow G''\rightarrow 0
$$ 
une suite exacte de schémas en $\BF$-vectoriels étales sur $k$. Comme la dualité de Cartier est exacte pour les schémas en groupes finis sur $k$ et comme le dual d'un schéma en groupes fini étale est connexe, on obtient une suite exacte de schémas en groupes connexes
$$
0\rightarrow G''^D\rightarrow G^D \rightarrow G'^D\rightarrow 0.
$$ 
Donc, d'après la proposition \ref{prop:multcarco} et le lemme \ref{lem:dual}, on obtient
$$
\Ch_k(G) = \Ch_k(G^D) = \Ch_k(G''^D) \Ch_k(G'^D)=\Ch_k(G'') \Ch_k(G').
$$
Ainsi, le caractère $\Ch_k$ est multiplicatif sur les schémas en $\BF$-vectoriels étales. 

Soit $G$ un schéma en $\BF$-vectoriels. Alors, en tant que schéma en groupes fini, comme $k$ est parfait, on a $G \cong G^{\circ}\times_k G^{\text{ét}}$ où $G^{\circ}$ est connexe et $G^{\text{ét}}$ est étale. Or, comme cette décomposition est fonctorielle, $G^{\circ}$ et $G^{\text{ét}}$ sont naturellement des schémas en $\BF$-vectoriels. La décomposition nous donne une surjection non canonique $G\rightarrow G^{\circ}$ et on déduit que
$$
\Ch_k(G) = \Ch_k(G^{\circ})\Ch_k(G^{\text{ét}}).
$$
Comme $\Ch_k$ est multiplicatif sur les schémas en $\BF$-vectoriels étales et sur les schémas en $\BF$-vectoriels connexes, on en déduit que $\Ch_k$ est multiplicatif, ce qui termine la démonstration de la proposition \ref{prop:multcar}. \qed

\subsection{Formule des caractères}
Dans cette partie on démontre le théorème \ref{thm:expcar}. On commencera par le démontrer dans le cas où $G$ est de type additif, puis par un dévissage et un argument de dualité, on le montrera en toute généralité.
\subsubsection{}
On commence par rappeler la classification des schémas en groupes finis sur $k$ de type additif (i.e. annulés par $V$). Soit $M$ un $k$-espace vectoriel de dimension finie. On peut lui associer un schéma en groupes sur $k$, annulé par $V$, en posant pour tout schéma affine $U=\Spec B$ sur $k$,
$$
M\otimes \BG_a (U) = M\otimes_k B,
$$
où la structure de groupe est déterminée par le groupe additif sous-jacent à l'anneau $B$. Notons que le Frobenius $F$ sur $\BG_a$ induit un Frobenius, toujours noté $F$ de $M\otimes \BG_a$. Si, de plus, $M$ est muni d'un opérateur $\sigma$-linéaire $F_M\colon M \rightarrow M$, alors $M\otimes \BG_a$ est muni d'un second Frobenius que l'on note $F_M$. Le foncteur des modules de Dieudonné sur $\W(k)$ annulés par $V$, dans les schémas en groupes finis de type additif, défini par $M\rightsquigarrow \ker(F-F_M)$ est alors une équivalence de catégories abéliennes (cf. \cite[IV, \S 3, 6.6-6.7]{dega}). Comme ces constructions sont fonctorielles en $M$, si $M$ est gradué sur $\BF^+$, alors $\ker(F-F_M)$ est naturellement un schéma en $\BF$-vectoriels sur $k$, puisque les Frobenius introduits respectent la graduation. On en déduit la proposition suivante :
\begin{proposition}\label{prop:dieuadd}
L'équivalence de catégories abéliennes
$$
\left \{\begin{array}c
G \text{ schéma en $\BF$-vectoriels sur k,}\\
\text{ de type additif} 
\end{array}\right \}\xleftrightarrow {} \left \{\begin{array}c
											\text{ $\W(k)$-modules de Dieudonné gradués}\\
											\text{annulés par $V$} 
											\end{array}\right \}
$$
définie par $G\rightsquigarrow D(G)$ a pour quasi-inverse le foncteur
$$
M\rightsquigarrow \ker\left [(F-F_M)\colon M\otimes \BG_a \rightarrow M\otimes \BG_a\right ].
$$
\end{proposition}
\qed
\subsubsection{}
En utilisant la proposition précédente, on démontre dans ce paragraphe le théorème \ref{thm:expcar} pour $G$ de type additif. 

Soit $G$ un schéma en $\BF$-vectoriels sur $k$ de type additif, $A$ sa $k$-algèbre de Hopf et $D(G)$ son module de Dieudonné. Soit $x_1,\dots, x_n$ une base homogène de $D(G)$ et $\chi_i \in \BF^+$ le caractère tel que $x_i\in D(G)_{\chi_i}$. Soit $I= \{1,2,\dots,n\}$, on a 
$$
\cha_k(D(G))= \sum_{i\in I}[\chi_i].
$$
Il suffit donc de montrer que
$$
\Cha_k(G) = \prod_{i\in I}\left (\sum_{k\in\{0,1,\dots, p-1\}}[\chi_i^k]\right ).
$$
 Soit $V = \{0,\dots, p-1\}^I$, l'ensemble des fonctions de $I$ à valeurs dans $\{0,\dots, p-1\}$. Alors, en échangeant la somme et le produit dans la relation précédente, on se ramène à montrer que
 \begin{equation}\label{hyp:rec2}
 \Cha_k(G) = \sum_{k\in V}\left [ \prod_{i\in I}\chi^{k(i)}_i\right ].
 \end{equation} 
D'après la proposition \ref{prop:dieuadd}, il existe $a_{i,j}\in k$, pour tout $i,j\in I= \{1,2,\dots,n\}$, tel que 
$$
A=\frac{k[x_1,\dots,x_n]} {\langle P_i\rangle_{i\in I}}, \text{ où pour tout }i\in I, \  P_i = x_i^p-\sum_{j\in I} a_{i,j}x_j.
$$
Donc la famille $\{x_1^{k(1)}\dots x_n^{k(n)}\}_{k\in V}$ est une base homogène de $A$ sur $k$. En particulier, pour $k\in V$ on a $x_1^{k(1)}\dots x_n^{k(n)}\in A_{\chi_1^{k(1)}\dots \chi_n^{k(n)}}$. Ceci démontre précisément (\ref{hyp:rec2}) et finit la démonstration du théorème \ref{thm:expcar} pour $G$ de type additif.
\qed
\begin{remark}
Comme les schémas en groupes finis étales sont de type additif, ceci montre en particulier le théorème pour $G$ étale.
\end{remark}
\subsubsection{}
Dans ce paragraphe on conclut la démonstration du théorème \ref{thm:expcar}. On commence par le cas où $V$ est nilpotent et on finit par un argument de dualité.

Soit $G$ un schéma en $\BF$-vectoriels sur $k$ tel que $V$ est nilpotent. Soit $M=D(G)$ son $\W(k)$-module de Dieudonné gradué. Si on pose pour tout entier $i\geqslant 0$, $M_i = V^iM$, on obtient une filtration décroissante de $M$ dont les quotients successifs sont annulés par $V$. Comme $V$ respecte la graduation, c'est une filtration décroissante graduée sur $\BF^+$ et elle définit donc une filtration décroissante $\{G_i\}_{i\geqslant 0}$ de sous-schémas fermés en $\BF$-vectoriels telle que les quotients successifs sont des schémas en $\BF$-vectoriels annulés par $V$. La multiplicativité de $\Cha_k$ et l'additivité de $\cha_k$ nous donnent donc
$$
\Cha_k(G) = \prod_{i\geqslant 0}\Cha(G_{i}/G_{i+1}),\ \text{ et }\ \cha_k(M) = \sum_{i\geqslant 0}\cha_k(M_{i}/M_{i+1}).
$$
Or, au paragraphe précédent on a montré que pour tout $i\geqslant 0$,
$$
\Cha_k(G_{i}/G_{i+1}) = \exp_{\BF}(\cha_k(M_{i}/M_{i+1})).
$$
Ainsi la multiplicativité de $\exp_{\BF}$ nous permet de conclure dans ce cas.

Il reste le cas où $V$ est un isomorphisme. Soit $G$ un schéma en $\BF$-vectoriels sur $k$ tel que $V$ est un isomorphisme et soit $D(G)$ son module de Dieudonné. Or, dans ce cas $F$ est nilpotent, ainsi $V$ est nilpotent sur $G^D$. Or, d'après le paragraphe \ref{sssub:dual}  on a $\Ch_k(G) = \Cha_k(G^D)$ et  d'après le lemme \ref{lem:dual} on a $\cha_k(D(G)) = \cha_k(D(G^D))$. On se ramène ainsi au cas où $V$ est nilpotent, ce qui termine la preuve du théorème.
\qed

\goodbreak
\section{Application aux modules $p$-divisibles}

Dans cette section on donne une application de la formule des caractères à la structure des points de torsions de certains modules $p$-divisibles. On utilise pour cela la théorie des schémas de Raynaud et leur structure (cf.  \cite{ray}). Après avoir rappelé ces résultats, on fera quelques rappels sur les cristaux des groupes $p$-divisibles et on étudiera les points de torsions des modules $p$-divisibles.

\subsection{Schémas de Raynaud}
Dans cette partie on garde les notations de la première section, en particulier $S$ désignera un schéma sur $D$. On fixe, de plus, un caractère primitif $\chi_1\in \BF^+$ et on note $\chi_i = \chi_1^{p^i}$ le $i$-ème caractère primitif.
\subsubsection{}
Soit $G$ un schéma en $\BF$-vectoriels sur $S$. Rappelons qu'on a une décomposition de l'idéal d'augmentation de $G$ de la forme
$$
\sI = \bigoplus_{\chi \in \BF^{\vee}}\sI_{\chi} \text{ pour } \sI_{\chi}= p_{\chi}( \sI).
$$

Pour $\chi$ et $\chi'$ des caractères primitifs, on peut restreindre la co-multiplication aux sommants de $\sI$ et obtenir $\Delta_{\chi,\chi'} \colon \sI_{\chi\chi'} \rightarrow \sI_{\chi}\otimes_S\sI_{\chi'}$. De même, pour la multiplication on obtient $\mu_{\chi,\chi'} \colon \sI_{\chi}\otimes_S\sI_{\chi'}\rightarrow \sI_{\chi\chi'} $. Comme précédemment, l'associativité et la co-associativité nous permettent de définir leurs itérations, pour tout $\chi_1,\dots, \chi_n \in \BF^{\vee}$, que l'on notera
$$
\Delta_{\chi_1,\dots \chi_n} \colon \sI_{\chi_1\dots\chi_n} \rightarrow \sI_{\chi_1}\otimes_S \dots\otimes_S\sI_{\chi_n}, \qquad \mu_{\chi_1,\dots, \chi_n} \colon \sI_{\chi_1}\otimes_S \dots \otimes_S\sI_{\chi_n}\rightarrow \sI_{\chi_1,\dots, \chi_n}.
$$

\begin{definition}
Soit $G$ un schéma en $\BF$-vectoriels sur $S$. On dit que $G$ est un \emph{$\BF$-schéma de Raynaud} si pour tout $\chi \in \BF^{\vee}$ le $\sO_S$-module $\sI_{\chi}$ est inversible. On dira de plus que $G$ est un \emph{$\BF$-schéma de Raynaud libre} si pour tout $\chi \in \BF^{\vee}$ le $\sO_S$-module $\sI_{\chi}$ est un $\sO_S$-module libre. S'il n'y a pas de confusion possible, nous dirons simplement que $G$ est un \emph{schéma de Raynaud}. 
\end{definition}
Il en résulte en particulier que $G$ est de rang $q$. En particulier, les $\BF_p$-schémas de Raynaud sont les schémas en groupes de Oort-Tate étudiés dans \cite{orta}. Une des différences dans ce cas est qu'un schéma en $\BF_p$-vectoriels localement libre de rang $p$ est toujours un $\BF_p$-schéma de Raynaud.

\subsubsection{}On introduit quelques notations supplémentaires pour énoncer le théorème de classification. Soit $G$ un $\BF$-schéma de Raynaud sur $S$. Alors on considère
\begin{itemize}
\item[$\bullet$] le morphisme $\Delta_i= \Delta_{\chi_i,\cdots, \chi_i} $ où le $i$-ème caractère primitif $\chi_i$ est considéré $p$ fois. Cette application peut être considérée canoniquement comme $\Delta_i \colon \sI_{i+1} \rightarrow \sI_i^p$ puisque $\chi_i^p=\chi_{i+1}$,
\item[$\bullet$] le morphisme $\mu_i= \mu_{\chi_i,\dots, \chi_i} $ où le $i$-ème caractère primitif $\chi_i$ est considéré $p$ fois. De même, on peut considérer cette application comme $\mu_i \colon \sI_i^p \rightarrow  \sI_{i+1}$,
\item[$\bullet$] la composée $w_i = \Delta_i \circ \mu_i$ est un élément de $\End(\sI_{i+1})$ et par la définition d'un schéma de Raynaud on peut considérer $w_i$ comme un élément de $\Gamma(S,\sO_S)$.
\end{itemize}
Raynaud a établi le lemme suivant (cf. \cite[Prop.1.3.1]{ray}) :
\begin{lemma}
Il existe $w\in pD^{\times}$ tel que pour tout $i$ avec $ 1\leqslant i \leqslant r$, $w_i$ soit l'image de $w$ par $D\rightarrow \Gamma(S,\sO_S)$. En particulier, $w$ est indépendant de $i$ et de $S$.
\end{lemma}
\qed

Dans la suite, on notera aussi $w$ l'image dans $\Gamma(S,\sO_S)$.
 De même que précédemment, toujours avec $\chi=\prod_j \chi_j^{a_j}$, on peut définir
 \begin{itemize}
\item[$\bullet$] le morphisme $\Delta_{\chi}= \Delta_{\chi_1,\dots, \chi_r} \colon \sI_{\chi} \rightarrow \sI_{\chi_1}^{\otimes a_1}\otimes_S \dots\otimes_S\sI_{\chi_r}^{\otimes a_r}$,
\item[$\bullet$] le morphisme $\mu_{\chi}= \mu_{\chi_1,\dots, \chi_r} \colon \sI_{\chi_1}^{\otimes a_1}\otimes_S \dots\otimes_S\sI_{\chi_r}^{\otimes a_r}\rightarrow \sI_{\chi} $,
\item[$\bullet$] la composée $w_{\chi} = \Delta_{\chi} \circ \mu_{\chi}$ qui est un élément de $\End(\sI_{\chi})$. On peut considérer $w_{\chi}$ comme un élément de $\Gamma(S,\sO_S)$.
\end{itemize}
Raynaud a également établi le lemme suivant  (cf. \cite[Prop 1.3.1]{ray}) :
\begin{lemma}
Soit $\chi\in \BF^{\vee}$. Alors $w_{\chi}$ est dans l'image de $D^{\times}\rightarrow \Gamma(S,\sO_S)$. En particulier il est inversible. 
\end{lemma}
\qed

\subsubsection{}On énonce maintenant le théorème principal de cette partie, établi par M. Raynaud en \cite[Prop.1.4.1]{ray}. 
 \begin{theorem}[Classification des schémas de Raynaud]\label{thm:claray}
Fixons $w\in pD^{\times}$. L'application définit sur l'ensemble des classes d'isomorphismes de $\BF$-schémas de Raynaud
$$
G \mapsto (\sI_{\chi_i},\Delta_i, \mu_i)_{1\leqslant i \leqslant r}
$$
est une bijection à valeurs dans l'ensemble des triplets $(\sI_{\chi_i},\Delta_i, \mu_i)_{1\leqslant i \leqslant r}$, constitués de 
\begin{itemize}
\item[$\bullet$] un système $(\sI_i)_{1\leqslant i \leqslant r}$ de classes d'isomorphismes de $\sO_S$-modules inversibles,
\item[$\bullet$] deux systèmes de morphismes
$$
\begin{cases}
(\Delta_i\colon \sI_{i+1}\rightarrow \sI_i^p)_{1\leqslant i \leqslant r}\\
(\mu_i\colon \sI_i^p \rightarrow \sI_{i+1})_{1\leqslant i \leqslant r} 
\end{cases}
$$
vérifiant pour tout $i$, tel que $1\leqslant i \leqslant r$, la relation $\mu_i \circ \Delta_i = w\cdot \id_{\sI_{i+1}}$ dans $\End(\sI_{i+1})$.
\end{itemize}
\end{theorem}
\qed

On retrouve pour $r=1$ la classification des schémas en groupes d'ordre premier $p$ établie dans \cite{orta}. 

En particulier, on obtient la description suivante des schémas de Raynaud libres, qui s'applique par exemple si $S$ est le spectre d'un anneau local.
\begin{corollary}\label{cor:localeq}
Fixons $w\in pD^{\times}$. Soit $G$ un $\BF$-schéma de Raynaud libre sur $S$. Alors $G$ est entièrement déterminé à isomorphisme près par la donnée de $r$ couples $(x_i,y_i)_{1\leqslant i \leqslant r}$ d'éléments de $\Gamma(S,\sO)$ tels que $x_iy_i = w$ pour tout $i$ avec $1\leqslant i \leqslant r$. De plus, $G$ est isomorphe à 
 \begin{equation}\label{eq:frmcan}
 \Spec \frac{\sO_S[z_1,\dots,z_r]}{\langle z_i^p-x_iz_{i+1}\rangle_{i=1,\dots,r}},
\end{equation}
 où l'indice est considéré modulo $r$. La co-multiplication est donnée par
 $$
 \Delta(z_i)= z_i \otimes 1 + 1 \otimes z_i+ \sum_{\chi\chi'=\chi_i}\frac{x_{i-h}\cdot \dots \cdot x_{i-1}} {w_{\chi}w_{\chi'}}(\prod_jz_j^{a_j})\otimes (\prod_jz_j^{a_j'}),
 $$ 
 pour $\chi=\prod_j \chi_j^{a_j}$, $\chi'=\prod_j \chi_j^{a_j'}$ et $h$ un entier dépendant de $\chi$ et $\chi'$. Le dual de Cartier $G^{D}$ est obtenu en échangeant les rôles de $x_i$ et $y_i$.
 
 De plus, les familles de couples $(x_i,y_i)_{1\leqslant i \leqslant r}$ et $(x_i',y_i')_{1\leqslant i \leqslant r}$,  tels que $x_iy_i=x_i'y_i'=w$ pour tout $i$ avec $1\leqslant i\leqslant r$, définissent des $S$-schémas en $\BF$-vectoriels isomorphes si et seulement s'il existe une famille d'unités $\lambda_i\in \Gamma(S,\sO_S)^{\times}$ telles que
 $$
 x_i'=\lambda_i^p x_i \lambda_{i+1}^{-1} \text{ et } y_i'=\lambda_i^{-p}y_i\lambda_{i+1} \text{ pour tout $1\leqslant i \leqslant r$.}
 $$

 \end{corollary}
 \qed
 \begin{definition}
Si $G$ est $\BF$-schéma de Raynaud libre sur $S$ de la forme (\ref{eq:frmcan}) on dira que $G$ est le $\BF$-schéma de Raynaud de paramètres $(x_i,y_i)_{1\leqslant i \leqslant r}$.
 \end{definition}
 
 \subsubsection{}
 On suppose maintenant que $p$ est localement nilpotent sur $S$. Soit $\sF$ un cristal en $\sO_{S_{\BF}/\Sigma}$-modules. Alors on dira que $\sF$ est \emph{spécial} si son caractère est de la forme
 $$
 \cha_S(\sF) = \sum_{\chi\in \BF^+}[\chi],
 $$
ou de manière équivalente, que ses composantes isotypiques  $\sF_{\chi}$ sont de rang $1$ sur $\sO_{S/\Sigma}/p\sO_{S/\Sigma}$ pour tout $\chi \in \BF^+$.

Soit $G$ un schéma en $\BF$-vectoriels sur $S$. Alors $G$ est un schéma de Raynaud si et seulement si son caractère est de la forme 
$$
\Cha_S(G) = 1 + \sum_{\chi\in\BF^{\vee}}[\chi].
$$
Ainsi, par la formule des caractères du théorème \ref{thm:expcarg}, on a la proposition suivante :
\begin{proposition}\label{prop:condray}
Soit $G$ un schéma en $\BF$-vectoriels sur $S$. Alors $G$ est un schéma de Raynaud si et seulement si $\BD(G)$ est spécial.
\end{proposition}
De plus, on montre que la condition pour être un schéma de Raynaud peut être affaiblie : 
\begin{corollary}
Soit $G$ un schéma en $\BF$-vectoriels sur $S$. Alors $G$ est un schéma de Raynaud si et seulement si pour tout $\chi \in \BF^+$, le faisceau localement libre $\sI_{\chi}$ est un faisceau inversible.
\end{corollary}
\begin{proof}
La réciproque est immédiate, on suppose donc que $G$ est un schéma en $\BF$-vectoriels sur $S$ tel que $\sI_{\chi}$ est un faisceau inversible pour tout $\chi \in \BF^+$. On veut montrer que $G$ est un schéma de Raynaud. 

D'après la formule des caractères, pour tout $\chi \in \BF^+$, il existe un entier positif $n_{\chi}$ tel que 
\begin{equation}\label{eq:ecr1}
\Cha_S(G) = \prod_{\chi\in \BF^+}\left ( 1+[\chi]+\dots +[\chi^{p-1}]\right )^{n_{\chi}} .
\end{equation}
Or, d'après notre hypothèse il existe pour tout $\chi\in \BF^{\vee}\backslash\BF^+$ un entier positif $a_{\chi}$ tel que
$$
\Cha_S(G) = \sum_{\chi\in \BF^+}[\chi]+\sum_{\chi \in \BF^{\vee}\backslash\BF^+}a_{\chi}[\chi].
$$
On compare cette expression avec le développement du produit dans (\ref{eq:ecr1}). On obtient pour tout $\chi \in \BF^+$ que $n_{\chi}\leqslant 1$ et que $n_{\chi}$ n'est pas nul ; donc $n_{\chi} = 1$. Ainsi, $G$ est un schéma de Raynaud.
\end{proof}

\subsection{Modules $p$-divisibles}
\subsubsection{}
Soit $S$ un $\W(\BF)$-schéma tel que $p$ soit localement nilpotent. Soit $X$ un groupe $p$-divisible sur $S$. Comme pour les groupes finis, on peut alors lui associer un cristal de Dieudonné covariant
$$
\BD(X) = \sExt^1_{S/\Sigma}(\underline{X^D}, \sO_{S/\Sigma})^{(\sigma^{-1})}.
$$
où $X^D$ est le dual de Cartier de $X$. C'est un cristal en $\sO_{S/\Sigma}$-modules localement libre (cf. \cite[Théorème 3.3.10]{bbm}). De plus, si $X$ est de hauteur $h$ sur $S$, alors $\BD(X)$ est localement libre de rang $h$. D'après \cite[Corollaire 3.3.5]{bbm}, on a une suite exacte de $\sO_S$-modules, fonctorielle en $X$ et en $S$,
\begin{equation}\label{eq:filhod}
0 \rightarrow \omega_{X^D} \rightarrow \BD(X)_S \rightarrow \Lie(X)\rightarrow 0,
\end{equation}
où $\omega_{X^D}$ sont les différentielles en la section unité de $X^D$ et $\Lie(X)$ l'algèbre de Lie de $X$.

Supposons que $S$ soit le spectre d'un corps parfait $k$. Soit $f\colon X \rightarrow X$ une isogénie entre groupes $p$-divisibles sur $k$. Alors $G=\ker(f)$ est par définition un schéma en groupes fini et localement libre sur $S$. Alors d'après \cite[Proposition 3.3.13]{bbm}, l'application
$$
\BD(X^D) \rightarrow \BD(G^D)
$$
est surjective et son noyau est $f\cdot \BD(X^D)$, où $f\colon \BD(X^D)\rightarrow \BD(X^D)$ est induit de $f$ par fonctorialité. On en déduit par dualité, en passant au module de Dieudonné et en utilisant le fait que $\BD(G)$ et $\BD(X)$ sont des cristaux, que
\begin{equation}\label{eq:tors}
\BD(G) \cong \BD(X) / f\cdot \BD(X).
\end{equation}
\subsubsection{}
On introduit quelques notations. Soit $K$ une extension finie de degré $n$ de $\BQ_p$ et $O_K$ son anneau d'entiers. On fixe $\pi \in O_K$ une uniformisante et on note $\kappa = O_K/(\pi)$ son corps résiduel. On note $f$ le degré de $\kappa$ sur $\BF_p$ et $e$ le degré de ramification de $K$ sur $\BQ_p$. On garde les notations précédentes, $\kappa^{\vee} = \Hom(\kappa^{\times}, \bar C^{\times})$ les caractères multiplicatifs de $\kappa$ et $\kappa^+ \subset \kappa^{\vee}$ les caractères primitifs. On fixe $S$ un schéma sur $\W(\kappa)$ tel que $p$ soit localement nilpotent.
\begin{definition}
Un \emph{$O_K$-module $p$-divisible} sur $S$ est un groupe $p$-divisible $X$ sur $S$ muni d'un morphisme d'anneaux
$$
\iota \colon O_K \rightarrow \End(X).
$$
\end{definition}
Dans la suite de ce paragraphe on fixe un $O_K$-module $p$-divisible $X$ sur $S$. On a une décomposition
$$
D\otimes_{\BZ_p}\W(\kappa) = \prod_{\chi \in \kappa}D_{\chi}.
$$
On note comme précédemment $p_{\chi}\colon D\otimes_{\BZ_p}\W(\kappa) \rightarrow D_{\chi}$ la projection pour tout $\chi \in \kappa^+$. On obtient des décompositions
\begin{equation}
\BD(X) = \bigoplus_{\chi \in \kappa^+}\BD(X)_{\chi},\ \Lie(X) = \bigoplus_{\chi \in \kappa^+}\Lie(X)_{\chi}, \ \omega_{X^D} = \bigoplus_{\chi \in \kappa^+}(\omega_{X^D})_{\chi},
\end{equation}
qui sont compatibles avec la suite exacte \ref{eq:filhod} car elle est fonctorielle en $X$.
\begin{definition}
Soit $\sF$ un $\sO_S$-module (resp. $\sO_{S/\Sigma}$-module) localement libre gradué sur $\kappa^+$, i.e.
$$
\sF = \bigoplus_{\chi \in \kappa^+}\sF_{\chi}.
$$
Alors on définit son caractère par 
$$
\cha_S(\sF) = \sum_{\chi \in \kappa^+}\rang_S(\sF_{\chi})[\chi]\in \BN[\kappa^+].
$$
\end{definition}
Cette définition ne porte pas à confusion avec la définition précédente de $\cha_S$ puisqu'il dépend de la nature de $\sF$. De plus, ces définitions sont compatibles, par exemple si $\sF$ est un cristal en $\sO_{S/\Sigma}$-modules
$$
\cha_S(\sF) = \cha_S(\sF_S).
$$
 Il est claire que ces caractères sont additifs. Ainsi, la filtration de Hodge $\ref{eq:filhod}$ nous donne la relation 
 $$
 \cha_S(\BD(X)) = \cha_S(\Lie(X)) + \cha_S(\omega_{X^D}).
 $$

On calcule maintenant le caractère d'un $O_K$-module $p$-divisible. Ce n'est pas un résultat profond mais plus un exercice que l'on peut résoudre de plusieurs façons ; on propose ici un calcul similaire à celui que l'on ferra pour les $O_D$-module $p$-divisibles. On note que $X[\pi]$ et $X[p]$ sont des schémas en $\kappa$-vectoriels.
\begin{proposition}\label{prop:divcar}
Soit $X$ un $O_K$-module $p$-divisible de hauteur $hn$ où $h$ est un entier positif. Alors
$$
\cha_S(\BD(X[\pi])) = h\left ( \sum_{\chi \in \kappa^+}[\chi]\right ), \ \cha_S(\BD(X)) = he\left ( \sum_{\chi \in \kappa^+}[\chi]\right ). 
$$
\end{proposition}
\begin{proof}
Comme précédemment, (cf. \ref{sssec:bla}) on peut supposer que $S = \Spec (k)$ où $k$ est un corps algébriquement clos de caractéristique $p$. 

On note $M= D(X)$, le module de Dieudonné de $X$, il est gradué sur $\kappa^+$, i.e.
$$
M = \bigoplus_{\chi\in \kappa^+}M_{\chi}.
$$
Les deux opérateurs $F$ et $V$ induisent $F \colon M_{\chi}\rightarrow M_{\chi^p}$ et $V \colon M_{\chi^p}\rightarrow M_{\chi}$ pour tout $\chi \in \kappa^+$. Comme l'opérateur $V$ est injectif, on en déduit que $\rang_{\W(k)}M_{\chi}$ est indépendant de $\chi \in \kappa^+$. Ainsi il existe un entier $a$ tel que
$$
\cha_k(M) = a\left ( \sum_{\chi \in \kappa^+}[\chi]\right ).
$$
Or, comme $X$ est de hauteur $hn=hef$, on en déduit que $a=he$, ce qui montre la deuxième égalité.

Montrons la première égalité. D'après (\ref{eq:tors}), on a $D(X[\pi]) = M/\pi M$, c'est un module de Dieudonné gradué sur $\kappa^+$. Pour tout $\chi \in \kappa^+$, on a un diagramme commutatif, dont tous les membres ont même rang sur $\W(k)$, 
\begin{equation}
\begin{tikzcd}
  M_{ \chi^p} \arrow[r,rightarrow,"V"] \arrow[d,rightarrow,"\pi"]&M_{ \chi} \arrow[d,rightarrow, "\pi"] \\
  M_{\chi^p}\arrow[r,rightarrow,"V"]& M_{\chi} .
\end{tikzcd}
\end{equation}
Comme les co-noyaux des flèches horizontales sont égaux, les co-noyaux des flèches verticales ont mêmes dimensions sur $k$ (on peut le déduire du lemme du serpent). Ainsi, comme précédemment,
$$
\cha_k(M/ \pi M) = h\left ( \sum_{\chi \in \kappa^+}[\chi]\right ),
$$
ce qui montre la première égalité.
\end{proof}
\subsubsection{}
On garde les notations du paragraphe précédent. Soit $d\geqslant 2$ un entier et $D$ l'algèbre à division d'invariant $1/d$ sur $K$. Alors $D$ contient $\widetilde K$, l'extension non-ramifié de $K$ de degré $d$ ; on note $O_{\widetilde K}$ son anneau des entiers. On note $\BF$ le corps résiduel de $D$, qui est le corps résiduel de $\widetilde K$ ; on notera $q$ le cardinal de $\BF$. Soit $O_D$ l'ordre maximal de $D$. On fixe un élément primitif $\Pi \in O_D$ tel que $\Pi^d = \pi$. On fixe $S$ un schéma sur $\W(\BF)$, que l'on supposera connexe, tel que $p$ soit localement nilpotent.
\begin{definition}
Un \emph{$O_D$-module $p$-divisible} sur $S$ est un groupe $p$-divisible $X$ sur $S$, muni d'un morphisme d'anneaux
$$
\iota \colon O_D \rightarrow \End(X).
$$
\end{definition}
En particulier, un $O_D$-module $p$-divisible est un $O_{\widetilde K}$-module $p$-divisible, ainsi les considérations du paragraphe précédent sont toujours valables. Ainsi, la proposition \ref{prop:divcar} nous donne le corollaire suivant :
\begin{corollary}
Soit $X$ un $O_D$-module $p$-divisible de hauteur $hnd^2$ sur $S$, pour $h$ un entier positif. Alors
$$
\cha_S(\BD(X[\pi]) = hd \left ( \sum_{\chi \in \BF^+}[\chi]\right ).
$$
\end{corollary}
\qed

Soit $f = \sum_{\chi \in \BF^+}a_{\chi}[\chi]\in \BN[\BF^+]$ un caractère. On notera
$$
f^{(p)} = \sum_{\chi \in \BF^+}a_{\chi}[\chi^p], \text{ et de même } f^{(q)} = \sum_{\chi \in \BF^+}a_{\chi}[\chi^q].
$$
Soit $X$ un $O_D$-module $p$-divisible de hauteur $hnd^2$ sur $S$, pour $h$ un entier positif. On considère le  caractère de $X[\Pi]$, qui est un schéma en $\BF$-vectoriels. L'additivité du caractère et le corollaire précédent impliquent alors que 
$$
\sum_{i=1}^d \cha_S(\BD(X[\Pi]))^{(q^i)} = \cha_S(\BD(X[\pi]).
$$
Ainsi, d'après la proposition \ref{prop:condray}, si $X[\Pi]$ est  $\BF$-schéma de Raynaud alors $h=1$. On veut expliciter une condition sur $X$ pour déterminer si $X[\Pi]$ est un schéma de Raynaud. On a le lemme suivant : 
\begin{lemma}
Soit $X$ un $O_D$-module $p$-divisible sur $S$. Alors on a la relation suivante : 
$$
\cha_S(\BD(X[\Pi]))^{(qp)} - \cha_S(\BD(X[\Pi]))^{(q)} = \cha_S(\Lie(X)) - \cha_S(\Lie(X))^{(q)}.
$$
\end{lemma}
\begin{proof}
Comme précédemment on peut supposer que $S = \Spec (k)$, le spectre d'un corps algébriquement clos de caractéristique $p$. Soit $M = D(X)$ le module de Dieudonné de $X$. Alors on a le diagramme commutatif suivant, dont les lignes et colonnes sont exactes
\begin{equation*}
\begin{tikzcd}
&0 \arrow[d,rightarrow,""] &0 \arrow[d,rightarrow,""] && \\
  0\arrow[r,rightarrow,""]& M^{(p)} \arrow[r,rightarrow,"V"] \arrow[d,rightarrow,"\Pi"]&M  \arrow[r,rightarrow,""] \arrow[d,rightarrow,"\Pi"]& \Lie(X) \arrow[d,rightarrow,"\Pi"]\arrow[r,rightarrow,""]& 0 \\
  0\arrow[r,rightarrow,""]& M^{(qp)} \arrow[r,rightarrow,"V"] \arrow[d,rightarrow,"\Pi"]&M^{(q)}  \arrow[r,rightarrow,""] \arrow[d,rightarrow,"\Pi"]& \Lie(X)^{(q)} \arrow[d,rightarrow,""]\arrow[r,rightarrow,""]& 0 \\
    & (M/\Pi M)^{(qp)} \arrow[r,rightarrow,"V"] \arrow[d,rightarrow,""]& (M/\Pi M)^{(q)} \arrow[r,rightarrow,""] \arrow[d,rightarrow,""]& \Lie(X[\Pi])^{(q)} \arrow[d,rightarrow,""]\arrow[r,rightarrow,""]& 0 \\
    & 0 & 0 & 0 & \\.
\end{tikzcd}
\end{equation*}
On a déjà remarqué que comme $V$ est injectif sur $M$, $\cha_k(M) = \cha_k(M^{(p)})$. On note $K = \ker (\Lie(X)\xrightarrow{\Pi}\Lie(X)^{(q)})$. Par le lemme du serpent  on a la suite exacte
$$
0 \rightarrow K \rightarrow (M/\Pi M)^{(qp)}\xrightarrow{V}(M/\Pi M)^{(q)} \rightarrow \Lie(X[\Pi])^{(q)} \rightarrow 0.
$$
Ainsi $\cha_k((M/\Pi M)^{(qp)}) - \cha_k((M/\Pi M)^{(q)}) = \cha_k(\Lie(X[\Pi])^{(q)}) - \cha_k(K)$. Mais de la dernière suite exacte verticale du diagramme, on a aussi $\cha_k(\Lie(X[\Pi])^{(q)}) - \cha_k(K) =\cha_k(\Lie(X)) - \cha_k(\Lie(X))^{(q)}$, d'où le lemme.
\end{proof}
On en déduit directement le théorème suivant, qui détermine quand les points de torsion d'un $O_D$-module formel forment un schéma de Raynaud.
\begin{theorem}
Soit $X$ un $O_D$-module $p$-divisible sur $S$ de hauteur $hnd^2$ pour $h$ un entier positif. Alors $X[\Pi]$ est un schéma de Raynaud si et seulement si $h=1$ et $\cha_S(\Lie(X)) = \cha_S(\Lie(X))^{(q)}$.
\end{theorem}
Cette condition est par exemple vérifiée pour les $O_D$-modules formels spéciaux de Drinfeld \cite{drin}, ou plus généralement pour les $O_D$-modules formels $r$-spéciaux considérés par Rapoport-Zink dans \cite{raza2} comme expliqué dans l'introduction. 

\newpage

\bibliographystyle{alpha}
\bibliography{ref}
\end{document}